\documentclass[reqno]{amsart}
\usepackage{bbm}
\usepackage{amsfonts}
\usepackage{mathrsfs}
\usepackage{hyperref}
\usepackage{amssymb}
\usepackage{CJK}
 \newtheorem{thm}{Theorem}[section]
 
 \newtheorem{prop}[thm]{Corollary}
 \newtheorem{open}[thm]{Open Question}
 
 \newtheorem{lem}[thm]{Lemma}

 \theoremstyle{remark}
 
 \numberwithin{equation}{section}
\DeclareMathOperator{\sign}{sign}

\begin{document}
\title[On the Odlyzko-Stanley enumeration problem]
  {  On the Odlyzko-Stanley enumeration problem and Waring's problem over finite fields}

\author{Jiyou Li}
\address{Department of Mathematics, Shanghai Jiao Tong University, Shanghai, P.R. China}
\email{lijiyou@sjtu.edu.cn}



\thanks{This work is supported by the National Science Foundation of China
(11001170).}

\begin{abstract}
We obtain an asymptotic formula on the Odlyzko-Stanley enumeration
problem. 
 Let $N_m^*(k,b)$ be the number of $k$-subsets
 $S\subseteq {\bf F}_p^*$ such that $\sum_{x\in S}x^m=b$.
 If $m<p^{1-\delta}$, then there is a constant
 $\epsilon=\epsilon(\delta)>0$ such that
 \begin{align*}
\left| N_m^*(k,b)-p^{-1}{p-1 \choose k}  \right|\leq
{p^{1-\epsilon}+mk-m \choose k}.
 \end{align*}

In addition, let $\gamma'(m,p)$ denote the distinct Waring's number
$(\!\!\!\!\mod p)$, the smallest positive integer $k$ such that
every integer is a sum of m-th powers of $k$-distinct elements
$(\!\!\!\!\mod p)$. The above bound implies that there is a constant
$\epsilon(\delta)>0$ such for any prime $p$ and any
$m<p^{1-\delta}$, if $\epsilon^{-1}<(e-1)p^{\delta-\epsilon}$, then
$$\gamma'(m,p)\leq \epsilon^{-1}.$$



\end{abstract}

\maketitle \numberwithin{equation}{section}
\newtheorem{theorem}{Theorem}[section]
\newtheorem{lemma}[theorem]{Lemma}
\newtheorem{example}[theorem]{Example}
\allowdisplaybreaks

\section{Introduction}
Let $p$ be an odd prime, and let ${\bf F}_p$ be the prime field of
order $p$. Let $m$ be a positive integer and $b$ be an element in
${\bf F}_p$. Let $N_m^*(b)$ be the number of subsets $S\subseteq
{\bf F}_p^*$ with the property that
$$\sum_{x\in S}x^m=b.$$
If $S$ is the empty set, we set $\sum_{x\in {\o}}x^m=0$.  For
details of this problem we refer to \cite{St}. It was shown by
Odlyzko-Stanley \cite{OS} that
 \begin{align}\label{1} \left |N_m^*(b)-2^{p-1}p^{-1}\right|\leq
e^{O(m\sqrt{p}\log{p})}.
\end{align}
  This bound can
be improved to a sharper bound
 \begin{align}\label{11} \left
|N_m^*(b)-2^{p-1}p^{-1}\right|\leq \frac {4}{\sqrt{2\pi}
}e^{m\sqrt{p}\log{p}}. \end{align}
 Moreover,  if
${\bf F}_p^*$ is replaced by ${\bf F}_q^*$, the multiplication group
of a finite field of order $q$ and of characteristic $p$, then
\begin{align}\label{2}\left |N_m^*(b)-2^{q-1}q^{-1}\right|\leq
\frac {4p}{\sqrt{2\pi} q}e^{(m\sqrt{q}+q/p)\log{q}}.
 \end{align}
These bounds follow directly from several counting formulas obtained
by Zhu-Wan \cite{W}. Their proof combines the techniques of Gauss
sums, Jacobi sums and a new sieving argument. More precisely, let
$N_m^*(k,b)$ be the number of $k$-subsets
 $S\subseteq {\bf F}_q^*$
 such that $\sum_{x\in S}x^m=b.$. 
 They proved that
\begin{align}\label{3}\left | N_m^*(k,b)- q^{-1}{q-1 \choose k}  \right|\leq
2q^{-1/2}{ m\sqrt{q}+q/p+k  \choose k}.
\end{align}
Note that $N_m^*(b)=\sum_{k=0}^q N_m^*(k,b)$ and  it is sufficient
to consider the case $|S|\leq(q-1)/2$ by symmetry. Hence (\ref{11})
and (\ref{2}) follow from (\ref{3}) directly.

Note that all above bounds are nontrivial only when
$n<p^{{1/2}-\epsilon}$. Heath-Brown, Konyagin and Shparlinski
\cite{HK,KS} improved this restriction to $n<p^{{2 \over
3}-\epsilon}$. Precisely, they obtain
 $$\left |N_m^*(b)-2^{p-1}p^{-1}\right|\leq
\left\{
\begin{array}{ll}
  e^{O(m{p}^{1/2}\log{p})}, \ \  m\leq {p}^{1/3};\\
  e^{O(m^{5/8}{p}^{5/8}\log{p})}, \ \  p^{1/3}\leq m\leq p^{1/2};\\
  e^{O(m^{3/8}{p}^{3/4}\log{p})},\ \   {p}^{1/2}\leq m\leq p^{2/3}.  \\
    \end{array}
    \right.
    $$
    Their proof relies on the monomial exponential sum bound
 $$\left |\sum_{x\in {\bf F}_p^*}e_p(ax^m) \right|\ll
\left\{
 \begin{array}{ll}
  m{p}^{1/2}, \ \  m\leq {p}^{1/3};\\
     m^{5/8}{p}^{5/8}, \ \  p^{1/3}\leq m\leq p^{1/2};\\
    m^{3/8}{p}^{3/4},\ \   {p}^{1/2}\leq m\leq p^{2/3};\\
    \end{array}
    \right.
    $$ for any integer $a$ with $p\nmid a$, where $e_p(x)=e^{2\pi i x/p}$ is the additive
    character on ${\bf F}_p$. Cochrane and Pinner
  \cite{CP} made explicit this bound to that
  $$\left |\sum_{x\in {\bf F}_p^*}e_p(ax^m) \right|\leq
\left\{
 \begin{array}{ll}
  m{p}^{1/2}, \ \  m\leq 3{p}^{1/3};\\
     \lambda m^{5/8}{p}^{5/8}, \ \  3p^{1/3}\leq m< p^{1/2};\\
    \lambda m^{3/8}{p}^{3/4},\ \   {p}^{1/2}\leq m< \frac 13p^{2/3};\\
    \end{array}
    \right.$$
      where $\lambda$ can be chosen to be $2/\sqrt[3]{4}\approx1.51967$.

When $m$ is large, Bourgain and Konyagin \cite{B1, BGK, BK} obtained
a celebrated nontrivial bound for a large kind of subgroups. Let $H$
be a subgroup of ${\bf F}_p^*$. Suppose $|H|>p^{\delta}$, then there
exits a constant ${\delta'}>0$ such that for any integer $a$ with
$p\nmid a$,
 \begin{align}\label{4}
  \bigg|\sum_{x\in H}e_p(ax) \bigg|<
  |H|^{1-\delta'}.
 \end{align}
For instance, Bourgain and Garaev proved in \cite{BG} that if
$\delta>1/4$, then one can take $\delta'=0.000015927+o(1)$. Taking
$H=\{x^m, x\in {\bf F}_p^*\}$ and following the same argument of
Konyagin and Shparlinski \cite{KS}, the above bound immediately
implies that if $m<p^{1-\delta}$, then
 \begin{align}\label{5}
 N_m^*(b)=2^{p-1}p^{-1}+e^{O({p^{1-\epsilon}})},
 \end{align}
 where $\epsilon=\epsilon(\delta)$ is a positive constant.This is a
 significant improvement of (\ref{1}).

  In this paper, by using the above bound and a distinct coordinate
  sieve argument, we first consider the subset sum problem over
   $H\subseteq {\bf F}_p^*$ and thus obtain a new counting formula
   via a combinatorial argument.  It gives a more precise bound on the
   number $N_m^*(k,b)$ for $m<p^{1-\delta}$ and suitable $k$.
  It is proved in this paper that
 \begin{thm}\label{thm1.0}Let $N_m^*(k,b)$ be the number of $k$-subsets
 $S\subseteq {\bf F}_p^*$ such that $\sum_{x\in S}x^m=b$.
 If $m<p^{1-\delta}$, then there is a constant
 $0<\epsilon=\epsilon(\delta)<\delta$ such that
 \begin{align}\label{6}
\left| N_m^*(k,b)-p^{-1}{p-1 \choose k}  \right|\leq
{p^{1-\epsilon}+mk-m \choose k}.
 \end{align}
  \end{thm}

\begin{prop} Suppose that $p, m, s, \delta, \epsilon$ are as in Theorem 1.1. If there is a constant $0<c<1$ such that
$-\frac 1 {\log{c}}\log{p}<k<cp^{\delta}-p^{\delta-\epsilon}$, then
the equation
\begin{align*}x_1^m+x_2^m+\cdots+x_k^m=b, \ x_i\in {\bf F}_p^*, \
x_i\neq x_j, \ i\ne j
 \end{align*}
 has at least a solution. In particular, if
 $\epsilon^{-1}<k<(e-1)p^{\delta-\epsilon}$, then the above equation
has a solution.
 \end{prop}

Note that this is a constant lower bound. This corollary has direct
application to the subset version of Waring's number mod $p$. We
first recall the  definition of ordinary Waring's number.  Let
$\gamma(m,p)$ denote Waring's number $(\!\!\!\!\mod p)$, the
smallest positive integer $k$ such that every integer is a sum of
m-th power $(\!\!\!\!\mod p)$.   This number has been thoroughly
studied. Note that we can always assume that $m< (p-1)/2$. The first
bound
$$\gamma(m,p)\leq m$$ for any prime $p$ was proved by Cauchy in
1813, as reported in \cite{AT}. Dozens of papers, for instance,
\cite{I, SHE, D1, D2, DT, H, K, C, WV}, studied Waring's number mod
a prime number, and generally, Warng's number mod an integer,
Waring's number over finite fields, p-adic integers and a general
commutative ring. We refer to \cite{CCP} for the previous results of
this problem.

%

The recent progress obtained by Ciper, Cocharane and Pinner
\cite{CCP} states that for any $\epsilon>0$ there is a constant
$c(\epsilon)$ such that if $\phi(s)\geq 1/\epsilon $ then
$$\gamma(m,p)\leq c(\epsilon)m^{\epsilon},$$
where $s=(p-1)/m$ and  $\phi$ is the Euler's totient function. By
the bound of Bourgain and Konyagin, and by a similar argument of
Konyagin and Shparlinski \cite{KS}, one can easily get
\begin{prop}\label{prop1.3}
 There is an absolute constant $C>0$ such that for $m<p^{1-\delta}$,
$$\gamma(m,p)\leq C^{1/\delta}.$$
\end{prop}
Cochrane and Cipra \cite{CC} showed that $C$ can be chosen to be 4
and $\gamma(m,p)\ll 4^{1/\delta}$.

  We now consider a stronger version of Waring's number, namely,
  the distinct or subset version of Waring's number. Let $\gamma'(m,p)$ denote the distinct Waring's
number $(\!\!\!\!\mod p)$, the smallest positive integer $k$ such
that every integer is a sum of m-th power of $k$ distinct elements
$(\!\!\!\!\mod p)$. Note that there are big differences between the
two Waring's numbers $\gamma(m,p)$ and $\gamma'(m,p)$. For example,
$\gamma'(m,p)$ does not exist $k$ is too large.
\begin{prop}There is a constant $\epsilon(\delta)>0$ such that for
any prime $p$ and any $m<p^{1-\delta}$, if
$\epsilon^{-1}<(e-1)p^{\delta-\epsilon}$, then we have
$$ \gamma'(m,p)<\epsilon^{-1}.$$
\end{prop}

Obviously $\gamma(m,p)\leq\gamma'(m,p)$ and thus this bound implies
Corollary \ref{prop1.3}, the known constant bound for ordinary
Waring's number.

Now we turn to the case for finite fields. Let ${\bf F}_q$ be the
finite field of order $q$ and of characteristic $p$. Let
$\gamma(m,q)$ denote the Waring's number in ${\bf F}_q$, the
smallest positive integer $k$ such that every element in ${\bf
F}_q^*$ is a sum of m-th power in ${\bf F}_q$. The work of A.
Winterhof \cite{WV} shows that
$$\gamma(m,q)\ll\frac{\log q}{\log p}m^{\log p/\log q}\log m$$
and J. Cipra  \cite{C} improved this bound to
$$\gamma(m,q)\ll\frac{\log q}{\log p}m^{\log p/\log q}.$$
Recently, Cochrane and Cipra \cite{CC} proved that
$$\gamma(m,q)\leq 633(2m)^{\frac{\log 4}{\log p-\log m}},$$  provided $m<p$ and $\gamma(m,q)$
exists.

 Similarly let $\gamma'(m,q)$ denote the distinct Waring's
number over ${\bf F}_q$, the smallest positive integer $k$ such that
every element in ${\bf F}_q$ is a sum of m-th power of distinct
elements in ${\bf F}_q^*$. Clearly $\gamma(m,q)\leq\gamma'(m,q)$.
The bound (\ref{3}) given by Zhu-Wan can improve the above bound for
 Waring's number over finite fields. Using (\ref{3}), Zhu
and Wan obtained:
\begin{prop}\cite{W} There is an effectively computable absolute constant $0 <
c < 1$ such that if $m<c\sqrt{q}$  and $6\ln{q} <k<(q-1)/2$ then
$N_m^*(k, b)
> 0$ for all $b\in {\bf F}_q$.
\end{prop}
This certainly implies a sharper bound at some cases:
\begin{prop} There is a constant $c>0$ such that if $m<c\sqrt{q}$
$$\gamma(m,q)\leq\gamma'(m,q)<\lfloor 6\ln{q}\rfloor+1.$$
\end{prop}
This paper is organized as follows. Proof of the main result will be
given in Section 3 and a distinct coordinate sieving method will be
introduced briefly in Section 2.

{\bf Notations}. For $x\in\mathbb{R}$, let  $(x)_0=1$
 and $(x)_k=x (x-1) \cdots (x-k+1)$
for $k\in $ $\mathbb{Z^+}$. For $k\in \mathbb{N}$, ${x \choose k}$
is the binomial coefficient defined by ${x \choose k}=\frac
{(x)_k}{k!}$.

\section{A distinct coordinate sieving formula}
In this section we introduce a sieving formula discovered by Li-Wan
\cite{LW2}, which significantly improves the classical
inclusion-exclusion sieve in many interesting cases. We cite it here
without any proof. For details and related applications please refer
to \cite{LW2,LW3}.

Let $D$ be a finite set, and let $D^k$ be the Cartesian product of
$k$ copies of $D$. Let $X$ be a
subset of $D^k$. 
Define $\overline{X}=\{(x_1,x_2,\cdots,x_k)\in X \ | \ x_i\ne x_j,
\forall i\ne j\}.$
Let $f(x_1,x_2,\dots,x_k)$ be a complex valued function defined over
$X$ and
$$F=\sum_{x \in \overline{X}}f(x_1,x_2,\dots,x_k).\label{1.00}$$

 Let $S_k$ be the symmetric group on $\{1,2,\cdots, k\}$.
Each permutation $\tau\in S_k$ factorizes uniquely 
as a product of disjoint cycles and  each
  fixed point is viewed as a trivial cycle of length $1$.
   Two permutations in $S_k$ are conjugate if and only if they
have the same type of cycle structure (up to the order).
 For $\tau\in S_k$, define the sign of $\tau$ to
  $\sign(\tau)=(-1)^{k-l(\tau)}$, where $l(\tau)$ is the number of
 cycles of $\tau$ including the trivial cycles.
 For a permutation $\tau=(i_1i_2\cdots i_{a_1})
  (j_1j_2\cdots j_{a_2})\cdots(l_1l_2\cdots l_{a_s})$
  with $1\leq a_i, 1 \leq i\leq s$, define
  \hskip 1.0cm
  \begin{align} \label{1.1}
     X_{\tau}=\left\{
(x_1,\dots,x_k)\in X,
 x_{i_1}=\cdots=x_{i_{a_1}},\cdots, x_{l_1}=\cdots=x_{l_{a_s}}
 \right\}.
\end{align}
 Similarly, for $\tau \in S_k$,  define $F_{\tau}=\sum_{x \in
X_{\tau} } f(x_1,x_2,\dots,x_k). $ Now we can state our sieve
formula.   We remark that there are many other interesting
corollaries of this formula. For interested reader we refer to
\cite{LW2}.

\begin{thm} \label{thm1.0}Let $F$ and $F_{\tau}$ be defined as above. Then   \begin{align}
  \label{1.5} F=\sum_{\tau\in S_k}{\sign(\tau)F_{\tau}}.
    \end{align}
 \end{thm}

Note that the symmetric group $S_k$ acts on $D^k$ naturally by
permuting coordinates. That is, for $\tau\in S_k$ and
$x=(x_1,x_2,\dots,x_k)\in D^k$,  $\tau\circ
x=(x_{\tau(1)},x_{\tau(2)},\dots,x_{\tau(k)}).$
  A subset $X$ in $D^k$ is said to be symmetric if for any $x\in X$ and
any $\tau\in S_k$, $\tau\circ x \in X $.
%
%
%
 For $\tau\in S_k$, denote by $\overline{\tau}$
 the conjugacy class determined by $\tau$ and it can
also be viewed as the set of permutations conjugate to $\tau$.
Conversely, for given conjugacy class $\overline{\tau}\in C_k$,
denote by $\tau$ a representative permutation of this class. For
convenience we usually identify these two symbols.

In particular,  if  $X$ is symmetric and $f$ is a symmetric function
under the action of $S_k$,  we then have the following simpler
formula than (\ref{1.5}).
\begin{prop} \label{thm1.1} Let $C_k$ be the set of conjugacy  classes
 of $S_k$.  If $X$ is symmetric and $f$ is symmetric, then
 \begin{align}\label{7} F=\sum_{\tau \in C_k}\sign(\tau) C(\tau)F_{\tau},
  \end{align} where $C(\tau)$ is the number of permutations conjugate to
  $\tau$.
\end{prop}

For the purpose of our proof, we will also need a combinatorial
formula.  A permutation $\tau\in S_k$ is said to be of type
$(c_1,c_2,\cdots,c_k)$ if $\tau$ has exactly $c_i$ cycles of length
$i$.  Note that $\sum_{i=1}^k ic_i=k$. Let $N(c_1,c_2,\dots,c_k)$ be
the number  of permutations in $S_k$ of type $(c_1,c_2,\dots,c_k)$
and it is well-known  that
$$N(c_1,c_2,\dots,c_k)=\frac {k!} {1^{c_1}c_1! 2^{c_2}c_2!\cdots k^{c_k}c_k!}.$$

\begin{lem} \label{lem2.6}
Define the generating function
\begin{align*}C_k(t_1,t_2,\dots,t_k)= \sum_{\sum
ic_i=k} N(c_1,c_2,\dots,c_k)t_1^{c_1}t_2^{c_2}\cdots t_k^{c_k}.
 \end{align*}
If $t_1=t_2=\cdots=t_k=q$, then we have
\begin{align} \label{6.2}
C_k(q,q,\dots,q) &=\sum_{\sum
ic_i=k} N(c_1,c_2,\dots,c_k)q^{c_1}q^{c_2}\cdots q^{c_k}\nonumber \\
&=(q+k-1)_k.
 \end{align}
\end{lem}

\section{Proof of Theorem 1.1}

Let $D\subseteq {\bf F}_p^*$ be a nonempty subset of cardinality
$n$.  Let $\chi_a=e_p(ax)=e^{2\pi i a x/p}$ be an additive character
over ${\bf F}_p$ and $\chi_0$ be the principal character sending
each element in ${\bf F}_p$ to 1. Denote by  $\widehat{{\bf F}}_p$
is the group of additive characters of ${\bf F}_p$. Define
$\Phi(D)=\max_{\chi\in \widehat{{\bf F}}_p, \chi\ne \chi_0}\left
|\sum_{a\in D}\chi{(a)}\right |.$   Let $N(k, b, D)$ be the number
of $k$-subsets $S\subseteq D$ such that $\sum_{x\in S}x=b$.
In the following lemma we will give an asymptotic bound on $N(k, b,
D)$ when $\Phi({D})$ is small compared to $n=|D|$.

 \begin{lem}\label{lem1.1}Let $N(k, b, D)$ be defined as above. Then
$$\left| N(k, b, D)-p^{-1}{n \choose k}\right|\leq
{\Phi(D)+k-1 \choose k}.$$
  \end{lem}

\begin{proof}
Let $X=D^k=D\times D \times \cdots \times D$ be the Cartesian
product of $k$ copies of $D$.
 Let $  \overline{X} =\left\{ (x_1,x_2,\dots,x_{k} )\in D^k \mid
 x_i\not=x_j,~ \forall i\ne j\} \right\}.$ It is clear that $|X|=n^k$ and
$|\overline{X}|=(n)_k$. Then

\begin{align*}
k!N(k, b, D)&={p^{-1}} \sum_{(x_1, x_2,\dots x_k) \in \overline{X}}
\sum_{\chi\in \widehat{{\bf F}}_p}\chi(x_1+x_2+\cdots +x_k-b)\\
&={p^{-1}} (n)_k+p^{-1} \sum_{\chi\ne \chi_0}\sum_{(x_1,
x_2,\cdots x_k) \in\overline{X}}\chi(x_1)\chi(x_2)\cdots \chi(x_k)\chi^{-1}(b)\\
&={p^{-1}}  {(n)_k}+{p^{-1}} \sum_{\chi\ne
\chi_0}\chi^{-1}(b)\sum_{(x_1,x_2,\dots x_k)
\in\overline{X}}\prod_{i=1}^{k} \chi(x_i).
\end{align*}
For $\chi\ne \chi_0$, let $f_{\chi}(x)=
f_{\chi}(x_1,x_2,\dots,x_{k})= \prod_{i=1}^{k}\chi(x_i),$ and for
  $\tau\in S_k$ let
$$F_{\tau}(\chi)=\sum_{x\in X_{\tau}}f_{\chi}(x)=\sum_{x \in X_{\tau}}\prod_{i=1}^{k} \chi(x_i),$$
where $X_{\tau}$ is defined as in (\ref{1.1}). Obviously $X$ is
symmetric and $f_{\chi}(x_1,x_2,\dots,x_{k})$ is normal on $X$.
Applying (\ref{7}) in Corollary \ref{thm1.1},
  \begin{align*}
 k!N(k, b, D)&={p^{-1}} {(n)_k}+{p^{-1}} \sum_{\chi\ne \chi_0}\chi^{-1}(b) \sum_{\tau\in
C_{k}}\sign(\tau)C(\tau) F_{\tau}(\chi),
  \end{align*}
 where $C_{k}$ is the set of conjugacy classes
 of $S_{k}$, $C(\tau)$ is the number of permutations conjugate to $\tau$, and
  \begin{align*}
F_{\tau}(\chi)&=\sum_{x \in X_{\tau}}\prod_{i=1}^{k} \chi(x_i)\\
&=\sum_{x \in X_{\tau}}\prod_{i=1}^{c_1} \chi(x_i)\prod_{i=1}^{c_2}
\chi^2(x_{c_1+2i})\cdots\prod_{i=1}^{c_k} \chi^k(x_{c_1+c_2+\cdots+k i})\\
 &=\prod_{i=1}^{k}(\sum_{a\in D}\chi^i(a))^{c_i}.
\end{align*}
By the definition of $\Phi(D)$,  $F_{\tau}(\chi) \leq
(\Phi(D))^{\sum_{i=1}^{k}c_i}$ and hence
 \begin{align*}
k!N(k, b, D)&\geq p^{-1}(n)_k-p^{-1}\sum_{\chi\ne \chi_0}
 \sum_{\tau\in C_{k}}C(\tau)(\Phi(D))^{\sum_{i=1}^{k}c_i}\\
 &={p^{-1}} {(n)_k}-{p^{-1}} (p-1)\sum_{\sum ic_i=k} \frac
{k!} {1^{c_1}c_1! 2^{c_2}c_2!\cdots k^{c_{k}}c_{k}!}
(\Phi(D))^{\sum_{i=1}^{k}c_i}\\
&={p^{-1}} {(n)_k}- (\Phi(D)+k-1)_k. \qedhere
\end{align*}
The last equality is from Lemma \ref{lem2.6} and the proof is
complete.
\end{proof}

%

This lemma together with the bound (\ref{4}) given by Bourgain and
Konyagin gives the following lemma.
\begin{lem}\label{lem3.3}
Choose $H=\{x^m, x\in {\bf F}_p^*\}$. Suppose that
$|H|=s>p^{\delta}$. Let $M(k,b)=M(k,b,H)$ be the number of
$k$-subsets $S\subseteq H$ such that $\sum_{x\in S}x=b$. Then we
have
$$\left|  {M(k, b)} -p^{-1}
{s\choose k}\right|\leq {s^{1-\delta'}+k-1 \choose k}.$$
\end{lem}
Next lemma is a counting formula, which allows us to
\texttt{"}lift\texttt{"} the solution of the subset sum problem in
the subgroup to the Odlyzko-Stanley enumeration problem.
\begin{lem}\label{lem1.3}Suppose $n\mid p-1$ and denote $s=(p-1)/n$. Then
 \begin{align*}{p-1 \choose k}&={s \choose k} {n\choose 1}^k+{s \choose k-1}(k-1){n \choose 1}^{k-2}{n \choose 2}+
\cdots\\+&{s\choose j}\sum_{i_1>0,i_2>0,\cdots,i_j>0,\sum_{t=1}^j
i_t=k} {n \choose i_1}{n \choose i_2}\cdots {n \choose i_j}+\cdots+
{s \choose 1}{n\choose k}.
 \end{align*}
 \begin{proof}
It is direct by a double counting argument. The left side counts the
number of $k$-subsets of $p-1$ balls. Divide $p-1$ balls into $s$
equal boxes with each of size $n$ and count the same number by two
steps. Choose boxes first and then choose the balls in the chosen
boxes. The number is exactly the right side.
 \end{proof}
\end{lem}

{ \bf Proof of Theorem 1.1} \ \  \ Choose $H=\{x^m, x\in {\bf
F}_p^*\}$.  We suppose that $m\mid p-1$ without loss of generality,
otherwise we can replace $m$ by $(m,p-1)$. Note that
$|H|=s=(p-1)/m>p^{\delta}$. Let $M(k,b)=M(k,b,H)$ be the number of
unordered solutions of the equation
\begin{align}\label{1.7}
x_1+x_2+\cdots+x_k=b,  \ x_i\in H,\  x_i\neq x_j, \ i\ne j.
 \end{align}
  By Lemma \ref{lem3.3} we have
$$\left|  {M(k, b)} -p^{-1}
{s\choose k}\right|\leq {s^{1-\delta'}+k-1 \choose k}.$$ 
 Recall $N_m^*(k,b)$ is also the number of unordered solutions of the diagonal
 equation  \begin{align} \label{1.8} x_1^m+x_2^m+\cdots+x_k^m=b,
 \ x_i\in {\bf F}_p^*,\  x_i\neq x_j, \ i\ne j.
 \end{align}
Similar to the proof of Lemma \ref{lem1.3}, any solution of
(\ref{1.7}) can be lifted to solutions of (\ref{1.8}). This counting
argument between (\ref{1.7}) and (\ref{1.8}) gives
 \begin{align*}N_m^*(k,b)&=M(k, b) {m\choose 1}^k+M(k-1, b)(k-1){m\choose 1}^{k-2}
 {m\choose 2}+
\cdots\\+&M(j, b)\sum_{i_1>0,i_2>0,\cdots,i_j>0,\sum_{t=1}^j i_t=k}
+{m \choose i_1}{m \choose i_2}\cdots {m \choose i_j}+\cdots M(1,
b){m\choose k}.
 \end{align*}
By Lemma \ref{lem1.3} this implies
\begin{align*}\left| N_m^*(k,b) -p^{-1}
{p-1\choose k}\right|&\leq {ps^{-\delta'}+mk-m \choose k}\\
&\leq{p^{1-\epsilon}+mk-m \choose k},
\end{align*}
where $\epsilon=\delta\delta'$ and the proof is complete. \qed

\begin{prop}\label{cor3.3} Suppose that $p, m, s, \delta, \epsilon$ are as in Theorem 1.1. If there is a constant $0<c<1$ such that
$-\frac 1 {\log{c}}\log{p}<k<cp^{\delta}-p^{\delta-\epsilon}$, then
the equation
\begin{align*}x_1^m+x_2^m+\cdots+x_k^m=b, \ x_i\in {\bf F}_p^*, \
x_i\neq x_j, \ i\ne j.
 \end{align*}
 has at least a solution. In particular, if we choose $c=ep^{-\epsilon}$, we then have a
 simpler condition $\epsilon^{-1}<k<(e-1)p^{\delta-\epsilon}$, which has a
constant lower bound.
 \end{prop}
\begin{proof}
By Theorem 1.1, to ensure $N_m^*(k,b)>0$ it is sufficient to have
 \begin{align*}
p^{-1}{p-1 \choose k}\geq{p^{1-\epsilon}+mk-m \choose k},
 \end{align*}
that is,
$$\frac {(p-1)_k} {(p^{1-\epsilon}+mk-m)_k}>p.$$
This leads to the following inequality
$$\frac {p} {p^{1-\epsilon}+mk}>p^{1/k}.$$
Take $0<c<1$ such that
$p^{1-\epsilon}+mk<p^{1-\epsilon}+p^{1-\delta}k<cp$ and we have
$c^{-1}>p^{1/k}$ and then $k>-\frac 1 {\log{c}}\log{p}$. Solve the
first inequality we get that $k<cp^{\delta}-p^{\delta-\epsilon}$. If
$c=ep^{-\epsilon}$, then the condition  becomes $k>\frac
{\log{p}}{\epsilon\log{p}-1}>\epsilon^{-1}$.
\end{proof}

\begin{open}
  Is it true that the  bound
  \begin{align*}\left| N_m^*(k,b) -p^{-1}
{p-1\choose k}\right|&\leq {p^{1-\epsilon}+k-1 \choose k}\\
\end{align*}
holds as $ (\ref{3})$ for any $m<p^{1-\delta}$?
\end{open}

 If this bound is true, then the bound (\ref{5}) will be strengthened by
 a significantly large error term and the bound in Corollary \ref{cor3.3} will
be improved.

%

\end{document}